\documentclass[12pt]{amsart}
\usepackage{latexsym, amsmath, amssymb, amsthm, mathrsfs}
\usepackage[alwaysadjust]{paralist}
\usepackage{caption}
\usepackage{subcaption}
\usepackage[T1]{fontenc}
\usepackage{lmodern}
\usepackage[backref,colorlinks=true,linkcolor=blue,urlcolor=blue, citecolor=blue]%
  {hyperref}
\usepackage[shortalphabetic]{amsrefs}
\usepackage[all]{xy}
\usepackage[T1]{fontenc}
\usepackage{mathptmx}
\usepackage{microtype}
\usepackage{framed}
\usepackage{tikz}
\usepackage{graphicx}

\usepackage[centering, includeheadfoot, hmargin=1in, vmargin=1in,
  headheight=30.4pt]{geometry}
\newtheorem{lemma}{Lemma}[section]
\newtheorem{theorem}[lemma]{Theorem}
\newtheorem{corollary}[lemma]{Corollary}
\newtheorem{proposition}[lemma]{Proposition}
\theoremstyle{definition}
\newtheorem{definition}[lemma]{Definition}
\newtheorem{remark}[lemma]{Remark}
\newtheorem{example}[lemma]{Example}

\newtheorem*{theoremA}{Theorem A}
\newtheorem*{theoremB}{Theorem B}
\newtheorem*{theoremC}{Theorem C}

\renewcommand{\theequation}%
{\arabic{section}.\arabic{lemma}.\arabic{equation}}

\newcommand{\CC}{\ensuremath{\mathbb{C}}} 
\newcommand{\NN}{\ensuremath{\mathbb{N}}} 
\newcommand{\PP}{\ensuremath{\mathbb{P}}} 
\newcommand{\QQ}{\ensuremath{\mathbb{Q}}} 
\newcommand{\RR}{\ensuremath{\mathbb{R}}} 
 
\newcommand{\sI}{\ensuremath{\kern -1pt \mathscr{I}\kern -2pt}} 
\newcommand{\sJ}{\ensuremath{\kern -2pt \mathscr{J}\kern -2pt}} 
\newcommand{\sO}{\ensuremath{\mathscr{O}}} 
\newcommand{\sU}{\ensuremath{\mathscr{U}}}

\renewcommand{\geq}{\geqslant}
\renewcommand{\leq}{\leqslant}

\DeclareMathOperator{\mult}{mult}

\DeclareMathOperator{\Supp}{Supp}
\DeclareMathOperator{\ord}{ord}

\newcommand{\equ}{\ensuremath{\,=\,}}
\newcommand{\deq}{\ensuremath{\stackrel{\textrm{def}}{=}}}

\newcommand{\st}[1]{\ensuremath{\left\{ #1 \right\} }}
\DeclareMathOperator{\Bbig}{Big}
\DeclareMathOperator{\Null}{Null}
\newcommand{\ybul}{\ensuremath{Y_\bullet}}
\newcommand{\dgeq}{\ensuremath{\,\geq\,}}  
\newcommand{\Bplus}{\ensuremath{\textbf{\textup{B}}_{+} }}
\newcommand{\Bminus}{\ensuremath{\textbf{\textup{B}}_{-} }}
\newcommand{\Bstable}{\ensuremath{\textbf{\textup{B}} }}
\newcommand{\dleq}{\ensuremath{\,\leq\,}} 
\newcommand{\origin}{\ensuremath{\textup{\textbf{0}}}}
\newcommand{\ei}{\ensuremath{\textup{\textbf{e}}_i}}
\newcommand{\eone}{\ensuremath{\textup{\textbf{e}}_1}}
\newcommand{\en}{\ensuremath{\textup{\textbf{e}}_n}}
\newcommand{\dsubseteq}{\ensuremath{\,\subseteq\,}}

\begin{document}

\title[Positivity and Newton--Okounkov bodies]{Positivity of line bundles and Newton-Okounkov bodies}

\author[A.~K\" uronya]{Alex K\" uronya}
\author[V.~Lozovanu]{Victor Lozovanu}

\address{Alex K\"uronya, Johann-Wolfgang-Goethe Universit\"at Frankfurt, Institut f\"ur Mathematik, Robert-Mayer-Stra\ss e 6-10., D-60325
Frankfurt am Main, Germany}
\address{Budapest University of Technology and Economics, Department of Algebra, Egry J\'ozsef u. 1., H-1111 Budapest, Hungary}
\email{{\tt kuronya@math.uni-frankfurt.de}}

\address{Victor Lozovanu, Universit\'a degli Studi di Milano--Bicocca, Dipartimento di Matematica e Applicazioni, 
Via Cozzi 53,, I-20125 Milano, Italy}
\email{{\tt victor.lozovanu@unimib.it}}

\maketitle


\section*{Introduction}

The aim  of this work is to characterize positivity (both local and global) of line bundles on complex projective varieties in terms of convex geometry 
via the theory of Newton--Okounkov bodies. We will provide descriptions of ample and nef divisors, and discuss the relationship between Newton--Okounkov 
bodies and Nakayama's $\sigma$-decomposition. 

Based on earlier ideas of Khovanskii's Moscow school and motivated by the  work of Okounkov \cite{Ok}, Kaveh--Khovanskii \cite{KKh} and Lazarsfeld--Musta\c t\u a \cite{LM} introduced Newton--Okounkov bodies to projective geometry, where they have been an object of interest ever since. Essentially, a refined book-keeping device encoding the orders of vanishing along subvarieties of the ambient space $X$, they provide a general framework for the study of the asymptotic behaviour of line bundles on projective varieties. 

The construction that leads to Newton--Okounkov bodies associates to a line bundle (or more generally, an $\RR$-Cartier divisor) on an $n$-dimensional variety a collection of compact convex bodies  $\Delta_{\ybul}(D)\subseteq\RR^n$ parametrized by certain  complete flags $\ybul$ of subvarieties. Basic properties of these  have been determined 
\cites{AKL,Bou1,LM}, and their behaviour on surfaces  \cites{KLM1,LM, LSS} and toric varieties  \cites{LM,PSU} has been discussed at length. We refer the reader to the 
above-mentioned sources for background information. 

A distinguishing property of the notion is that it provides a set of 'universal numerical invariants', since a result of Jow \cite{Jow} shows that for Cartier divisors $D$ 
and $D'$, $D$ is numerically equivalent to $D'$ precisely if the associated functions 
\[
 \text{Admissible flags $\ybul$ in $X$}\ \stackrel{\Delta_{\ybul}(D)}{\longrightarrow}\  \text{Convex bodies in $\RR^n$}
\]
agree. 

Turning this principle into practice, one can expect to be able to read off all sorts of numerical invariants of Cartier divisors --- among them asymptotic invariants 
like the volume or Seshadri constants --- from the set of Newton--Okounkov bodies of $D$. On the other hand,  questions about global properties of the divisor might arise; whether one can determine ampleness or 
nefness of a given divisor in terms of its Newton--Okounkov bodies. As we will see, the answer is affirmative.

Localizing this train of thought, local positivity of a divisor $D$ at a point $x\in X$ will be  determined by the function 
\[
 \text{Admissible flags centered at $x$} \ \stackrel{\Delta_{\ybul}(D)}{\longrightarrow}\ \text{Convex bodies in $\RR^n$}\ .
\]
In particular, one can aim at deciding containment of $x$ in various asymptotic base loci, or compute measures of local positivity in terms of these convex sets. 

In fact the authors have carried out the suggested analysis in the case of smooth surfaces \cite{KL}, where the answer turned out to be surprisingly complete. The current article can be rightly considered as a higher-dimensional generalization of \cite{KL}. 

In search for a possible connection between Newton--Okounkov bodies and positivity, let us start with the toy example of projective curves. For an $\RR$-Cartier 
divisor $D$ on a smooth projective curve $C$, one has 
\begin{eqnarray*}
 D \text{ nef } \Leftrightarrow \deg_C D \dgeq 0 & \Leftrightarrow & 0 \in \Delta_{P}(D) \text{ for some/any point $P\in C$}\ ,\\
 D \text{ ample } \Leftrightarrow \deg_C D \,>\,  0 & \Leftrightarrow & \Delta_\lambda \subseteq \Delta_{P}(D) \text{ for some/any point $P\in C$}, 
\end{eqnarray*}
where $\Delta_{\lambda}:=[0,\lambda]$ for some real number $\lambda>0$. 

Interestingly enough, the observation just made generalizes in its entirety for smooth projective surfaces. Namely, one has the following \cite{KL}*{Theorem A}: 
for a big $\RR$-divisor $D$ on a smooth projective surface $X$
\begin{eqnarray*}
\text{ $D$ is nef} &   \Leftrightarrow & \text{  for all $x\in X$ there exists a flag $(C,x)$ such that $(0,0)\in\Delta_{(C,x)}(D)$ }, \\
\text{  $D$ is ample} &  \Leftrightarrow & \text{for all $x\in X$  there exists a flag $(C,x)$ and $\lambda >0$ such that $\Delta_{\lambda}\subseteq \Delta_{(C,x)}(D)$ \  
}
\end{eqnarray*}
where $\Delta_{\lambda}$ denotes the standard full-dimensional simplex of size $\lambda$ in $\RR^2$. In higher dimensions we will also denote by $\Delta_{\lambda}\subseteq\RR^n$ the standard simplex of length $\lambda$. 

Our first  results are local versions of  the analogous statements in higher dimensions. 

\begin{theoremA}
Let $D$ be a big $\RR$-divisor on a smooth projective variety $X$ of dimension $n$, let $x\in X$. Then the following are equivalent.  
\begin{enumerate}
 \item $x\not\in \Bminus(D)$.
 \item There exists an admissible flag $\ybul$ on $X$ centered at $x$ such that the origin $\origin\in \Delta_{\ybul}(D)\subseteq\RR^n$.
 \item The origin $\origin\in \Delta_{\ybul}(D)$ for every admissible flag $\ybul$ on $X$ centered at $x\in X$.
\end{enumerate}
\end{theoremA}

\begin{theoremB}
 With notation as above, the following are equivalent.
 \begin{enumerate}
  \item $x\not\in \Bplus(D)$.
  \item There exists an admissible flag $\ybul$ on $X$ centered at $x$ with $Y_1$ ample such that $\Delta_\lambda\subseteq \Delta_{\ybul(D)}$ for some positive real number
  $\lambda$. 
  \item For every admissible flag $\ybul$ on $X$ there exists a real number $\lambda>0$ for which $\Delta_\lambda\subseteq \Delta_{\ybul(D)}$.
 \end{enumerate}
\end{theoremB}

These results will be proven below as Theorem~\ref{thm:main1}, and Theorem~\ref{thm:main2}, respectively. Making use of the connections between augmented/restricted 
base loci, we obtain the expected characterizations of nef/ample divisors as in Corollary~\ref{cor:nef} and \ref{cor:ample}. An interesting recent study of local positivity on surfaces was undertaken 
by Ro\'e \cite{Roe}, where the author introduces the concept of local numerical equivalence, based on the ideas developed in \cite{KL}.

Zariski decomposition is a basic tool in the theory of linear series on surfaces, which is largely responsible for the fact that Newton--Okounkov bodies are reasonably
well understood in dimension two; the polygonality of $\Delta_{\ybul}(D)$ in case of a smooth surface is a consequence of variation of Zariski decomposition \cite{BKS}
for instance (see \cite{KLM1}*{Section 2} for a discussion). 

Not surprisingly, the existence and uniqueness of Zariski decompositions is one of the main tools used in \cite{KL}.
Its relationship to Newton--Okounkov polygons on surfaces is particularly simple: if $D$ is a big $\RR$-divisor with the property that  the point $Y_2$ in the flag $\ybul$ is not contained in the support of the 
negative part of $D$, then $\Delta_{\ybul}(D)=\Delta_{\ybul}(P_D)$, where $P_D$ stands for the positive part of $D$. 

In dimensions three and above, the appropriate birational version of Zariski decomposition --- the so-called CKM decomposition --- only exists under fairly restrictive hypotheses, hence one needs substitutes 
whose existence is guaranteed while they still retain some of the favourable properties of the original notion. 

A widely accepted concept along these lines is Nakayama's divisorial Zariski decomposition or $\sigma$-decomposition, which exists for an arbitrary big $\RR$-divisor, but where the 'positive part' is only guaranteed
to be movable (see \cite{Nak}*{Chapter 3} or \cite{Bou2}). Extending the observation coming from dimension two, we obtain the following. 

\begin{theoremC}
Let $X$ be a smooth projective variety, $D$  a big $\RR$-divisor, $\Gamma$ a prime divisor,  
$Y_{\bullet}: Y_0=X\supseteq Y_1=\Gamma\supseteq\ldots  \supseteq Y_n=\{x\}$ and admissible flag on $X$. Then 
\begin{enumerate}
\item $\Delta_{Y_{\bullet}}(D) \ \subseteq \ (\sigma_{\Gamma}(D),0\ldots, 0)+\RR_+^n$,
\item $(\sigma_{\Gamma}(D),0\ldots, 0)\ \in \ \Delta_{Y_{\bullet}}(D)$, whenever $x\in \Gamma$ is a very general point.
\item $\Delta_{Y_{\bullet}}(D) = \nu_{Y_{\bullet}}(N_{\sigma}(D)) + \Delta_{Y_{\bullet}}(P_{\sigma}(D))$. Morever, $\Delta_{Y_{\bullet}}(D) =\Delta_{Y_{\bullet}}(P_{\sigma}(D))$, when $x\notin\textup{Supp}(N_{\sigma}(D))$.
\end{enumerate}
\end{theoremC}

The organization of the paper goes as follows: Section 1 fixes notation, and collects some preliminary information about asymptotic base loci and
Newton--Okounkov bodies. Sections 2 and 3 are devoted to the respective
proofs of Theorems A and B, while Section 4 describes the relationship between  Newton--Okounkov bodies and  Nakayama's $\sigma$-decomposition. 
 
\smallskip 
\paragraph*{\bf Acknowledgements}  We are grateful for helpful discussions to S\'ebastian Boucksom, Lawrence Ein, John Christian Ottem, Mihnea Popa, 
and Stefano Urbinati. Parts of this work were done while the authors attended the MFO workshop on Newton--Okounkov bodies, the Summer School in Geometry at University of Milano--Bicocca, 
and the RTG Workshop on Newton--Okounkov bodies at the University of Illinois at Chicago. We would like to thank the organizers of these events for the 
opportunity (Megumi Harada, Kiumars Kaveh, Askold Khovanskii; Francesco Bastianelli, Roberto Paoletti; Izzet Coskun and Kevin Tucker). 
 
 Alex K\"uronya was partially supported by the DFG-Forschergruppe 790 ``Classification of Algebraic Surfaces and Compact Complex Manifolds'', 
by the DFG-Graduier\-ten\-kol\-leg 1821 ``Cohomological Methods in Geometry'', and by the OTKA grants 77476 and 81203 of the Hungarian Academy of 
Sciences.

\section{Notation and preliminaries}

\subsection{Notation}
For the duration of this work  let $X$ be a smooth complex projective variety of dimension $n$ and $D$ be a Cartier divisor on $X$. An admissible flag of subvarieties
\[
Y_{\bullet} \ : \ X=Y_0 \supseteq Y_1\supseteq \ldots \supseteq Y_{n-1}\supseteq Y_n=\{\textup{pt.}\},
\]
is a complete flag with the property  that each $Y_i$ is an irreducible subvariety of codimension $i$ and smooth at the point $Y_n$. 
For an arbitrary point  $x\in X$,  we say that $Y_{\bullet}$ is \textit{centered at} $x$ whenever $Y_n=x$. The associated Newton--Okounkov body will be denoted by
 $\Delta_{Y_{\bullet}}(D)\subseteq \RR^n_+$. 

\begin{remark}
 Not all of our results require $X$ to be smooth, at points it would suffice to require $X$ to be merely a projective variety. As a rule though, we will not keep track of 
  minimal hypotheses. 
\end{remark}

\subsection{Asymptotic base loci}

Stable  base loci are fundamental invariants of linear series, however, as their behaviour is somewhat erratic (they do not respect numerical equivalence of divisors
for instance), other alternatives were in demand. To remedy the situation, Nakamaye came up with the idea of studying stable base loci of small perturbations. Based on this,
the influential paper  \cite{ELMNP1} introduced new asymptotic notions, the  \textit{restricted} and \textit{augmented base loci} of a big divisor $D$. 

The restricted base locus of a big $\RR$-divisor $D$ is defined as 
\[
\Bminus(D) \deq  \bigcup_{A} \textbf{B}(D+A)\ ,
\]
where the union is over all ample $\QQ$-divisors $A$ on $X$. This locus turns out to be  a countable union of subvarieties of $X$ (and one really needs a countable
union on occasion, see \cite{Les}) via  \cite{ELMNP1}*{Proposition~1.19}
\[
\Bminus(D) \equ \bigcup_{m\in\NN}\textbf{B}(D+\frac{1}{m}A)\ .
\]
The augmented base locus of $D$ is defined to be 
\[
\Bplus(D)\deq  \bigcap_{A}\textbf{B}(D-A),
\]
where the intersection is taken over all ample $\QQ$-divisors $A$ on $X$. It follows quickly from \cite{ELMNP1}*{Proposition 1.5} that $\Bplus(D)=\textbf{B}(D-\frac{1}{m}A)$ for all $m>>0$ and any fixed ample class $A$.

Augmented and restricted base loci satisfy various favorable properties; for instance  both $\Bplus (D)$ and $\Bminus(D)$ depend only on the numerical 
class of $D$, hence  are much easier to study (see \cite{ELMNP1}*{Corollary 2.10} and \cite{PAGII}*{Example 11.3.12}).  

Below we make a useful  remark regarding augmented/restricted base loci. The statement must be well-known to experts, as usual, we include it with 
proof for the lack of a suitable reference. 

 \begin{proposition}\label{prop:openclosed}
Let $X$ be a  projective variety, $x\in X$ an arbitrary point. Then 
\begin{enumerate}
 \item $B_+(x) \deq \st{\alpha\in N^1(X)_\RR\mid x\in \Bplus(\alpha)} \dsubseteq N^1(X)_\RR$ is closed, 
 \item $B_-(x) \deq \st{\alpha\in N^1(X)_\RR\mid x\in \Bminus(\alpha)} \dsubseteq N^1(X)_\RR$ is open, 
\end{enumerate}
 both with respect to the metric topology of $N^1(X)_\RR$.  
\end{proposition}
 
\begin{remark}
We point out  that unlike required   in \cite{ELMNP1}, one does not need the normality assumption on $X$ for \cite{ELMNP1}*{Corollary 1.6} to hold.  
\end{remark}

\begin{proof}
$(i)$ First we deal with the case of augmented base loci. Observe that it suffices to prove that 
\[
 B_+(x) \cap \Bbig(X) \subseteq \Bbig(X)
\]
is closed, since the big cone is open in the N\'eron--Severi space. 

We will show that whenever $(\alpha_n)_{n\in\NN}$ is a sequence of big $\RR$-divisor classes in $B(x)$ 
converging to $\alpha\in \Bbig(X)$,  then $\alpha\in B(x)$ as well. 

By \cite{ELMNP1}*{Corollary 1.6}, the class $\alpha$ has a small open neighbourhood $\sU$ in the big cone for which
\[
 \beta\in \sU \ \Longrightarrow \ \textup{\textbf{B}}_+(\beta)\subseteq \Bplus(\alpha)\ .
\]
If $x\in \Bplus(\alpha_n)$ for infinitely many $n\in \NN$, then since $\alpha_n \in U$ for $n$ large, 
we also have $x\in \Bplus(\alpha)$. \\
$(ii)$ Let $\alpha\in N^1(X)_\RR$ be arbitrary, and fix an $\RR$-basis $A_1,\dots,A_\rho$ of $N^1(X)_\RR$ consisting of ample divisor classes. 
Observe that $x\in \Bminus(\alpha)$ implies that $x\in \Bminus(\alpha+t_0\sum_{i=1}^{\rho}A_i)$ for some $t_0>0$ thanks to the definition 
of the restricted base locus. 

Since subtracting ample classes cannot decrease $\Bminus$, it follows that $x\in \Bminus(\alpha)$ yields $x\in \Bminus (\gamma)$ for all classes
of the form $\alpha+t_0\sum_{i=1}^{\rho}A_i - \sum_{i=1}^{\rho}\RR_{\geq 0}A_i$, which certainly contains an open subset of $\alpha\in N^1(X)_\RR$.
\end{proof}

\subsection{Newton--Okounkov bodies}

We start with  a sligthly different definition of Newton--Okounkov bodies; it has  already appeared in print  in \cite{KLM1}, and although it is an 
immediate consequence of \cite{LM}, a complete proof was first given in \cite{Bou1}{Proposition 4.1}.

\begin{proposition}[Equivalent definition of Newton-Okounkov bodies]\label{prop:definition}
Let $\xi\in\textup{N}^1(X)_{\RR}$ be a big $\RR$-class and $Y_{\bullet}$ be an admissible flag on $X$. Then
\[
\Delta_{Y_{\bullet}}(\xi) \ = \ \textup{closed convex hull of\ }\{ \nu_{Y_{\bullet}}(D) \ | \ D\in \textup{Div}_{\geq 0}(X)_{\RR}, D\equiv \xi\},
\]
where the valuation $\nu_{Y_{\bullet}}(D)$, for an effective $\RR$-divisor $D$, is constructed inductively as in the case of integral divisors.
\end{proposition}

\begin{remark}
Just as in the case of the original definition of Newton--Okounkov bodies, it becomes a posteriori clear that valuation vectors $\nu_{\ybul}(D)$ form a dense subset 
of 
\[
 \textup{closed convex hull of }\{ \nu_{Y_{\bullet}}(D) \ | \ D\in \textup{Div}_{\geq 0}(X)_{\RR}, D\equiv \xi\}\ ,
\]
hence it would  suffice to take closure in Proposition~\ref{prop:definition}. 
\end{remark}

The description of Newton--Okounkov bodies above is often  more suitable to use than the original one. For example, the following statement follows immediately  from it.

\begin{proposition}\label{prop:compute}
Suppose $\xi$ is a big $\RR$-class and $Y_{\bullet}$ is an admissible flag on $X$. Then for any $t\in [0,\mu(\xi, Y_1))$, we have
\[
\Delta_{Y_{\bullet}}(\xi)_{\nu_1\geq t} \ = \ \Delta_{Y_{\bullet}}(\xi-tY_1)\ + t\eone,
\]
where $\mu(\xi,Y_1)=\sup \{\mu>0|\xi-\mu Y_1 \textup{ is big}\}$ and $\eone=(1,0,\ldots ,0)\in \RR^n$.
\end{proposition}

This statement first appeared  in \cite{LM}*{Theorem~4.24} with the additional condition that $Y_1\nsubseteq \Bplus(\xi)$.

We  will need  a version of \cite{AKL}*{Lemma 8} for real divisors.

\begin{lemma}\label{lem:nested}
Let $D$ be a big $\RR$-divisor, $A$ an ample $\RR$-divisor, $Y_\bullet$  an admissible flag on $X$.  Then, for any real number 
$\epsilon > 0$, we have
\[
 \Delta_{Y_\bullet}(D) \subseteq \Delta_{Y_\bullet}(D+\epsilon A) \ ,
\]
and $\Delta_{Y_\bullet}(D)\equ \bigcap_{\epsilon>0}\Delta_{Y_\bullet}(D+\epsilon A)$.
\end{lemma}

\begin{proof}
For the first claim, since $A$ is an ample $\RR$-divisor, one can find an effective $\RR$-divisor $M\sim_{\RR}A$ with  $Y_n\notin\Supp(M)$. Then for any arbitrary 
effective divisor $F\sim_{\RR}D$ one has  $F+\epsilon M \equiv_{\RR} D+\epsilon A$ and $\nu_{\ybul}(F+M)=\nu_{\ybul}(F)$. Therefore
\[
\st{\nu_{Y_{\bullet}}(\xi) \mid \xi\in \textup{Div}_{\geq 0}(X)_{\RR}, D\equiv \xi} \subseteq %
\st{ \nu_{Y_{\bullet}}(\xi) \mid \xi \in \textup{Div}_{\geq 0}(X)_{\RR}, D+\epsilon A\equiv \xi}
\]
and we are done by Proposition~\ref{prop:definition}.

The equality of the second claim  is a  consequence of the previous inclusion and the continuity  of Newton--Okounkov bodies.
\end{proof}


\section{Restricted base loci} 

Our main goal here is to give a characterization of  restricted base loci in the language of Newton--Okounkov bodies.

\begin{theorem}\label{thm:main1}
Let $D$ be a big $\RR$-divisor on a smooth projective variety $X$ of dimension $n$, let $x\in X$. Then the following are equivalent.  
\begin{enumerate}
 \item $x\not\in \Bminus(D)$.
 \item There exists an admissible flag $\ybul$ on $X$ centered at $x$ such that $\origin\in \Delta_{\ybul}(D)\subseteq\RR^n$.
 \item The origin $\origin\in \Delta_{\ybul}(D)$ for every admissible flag $\ybul$ on $X$ centered at $x\in X$.
\end{enumerate}
\end{theorem}

Coupled with simple properties of restricted base loci we arrive at a precise description of big and nef divisors in terms of convex geometry. 

\begin{corollary}\label{cor:nef}
With notation as above the following are equivalent for a big $\RR$-divisor $D$.
\begin{enumerate}
 \item $D$ is nef.
 \item For every point $x\in X$  there exists an admissible flag $\ybul$ on $X$ centered at $x$ such that $\origin\in \Delta_{\ybul}(D)\subseteq\RR^n$.
 \item For every admissible flag $\ybul$,   one has $\origin\in \Delta_{\ybul}(D)$. 
\end{enumerate}
\end{corollary}
\begin{proof}
Immediate from Theorem~\ref{thm:main1} and  \cite{ELMNP1}*{Example 1.18}. 
\end{proof}

The essence of the proof of Theorem~\ref{thm:main1} is to connect the asymptotic multiplicity of $D$ at $x$ to a certain  function defined on the Newton-Okounkov body
of $D$. Before turning to the actual proof, we will quickly recall the notion of the \emph{asymptotic multiplicity}  or the \emph{asymptotic order of vanishing} of 
a $\QQ$-divisor $F$ at a point $x\in X$. 

Let $F$ be  an effective Cartier divisor on $X$, defined locally by the equation $f\in \sO_{X,x}$. Then \textit{multiplicity} of $F$ at $x$ is defined to be 
$\mult_x(F)=\textup{max}\{n\in\NN | f\in \mathfrak{m}^n_{X,x}\}$, where $\mathfrak{m}_{X,x}$ denotes the maximal ideal of the local ring $\sO_{X,x}$. If $|V|$ is a linear series, 
then the multiplicity of $|V|$ is defined to be 
\[
\mult_x(|V|) \deq \min_{F\in |V|}\{\mult_x(F)\}\ . 
\]
By  semicontinuity the above expression equals  the multiplicity  of a general element in $|V|$ at $x$. The \emph{asymptotic multiplicity} of a $\QQ$-divisor $D$ 
at $x$  is then defined to be
\[
\mult_x(||D||) \deq  \lim_{p\rightarrow \infty}\frac{\mult_x(|pD|)}{p}\ .
\]
The multiplicity at $x$ coincides with the order of vanishing at $x$, given in Definition~2.9 from \cite{ELMNP1}. In what follows we will talk about the multiplicity 
of a divisor, but the order of vanishing of a section of a line bundle. 

An important technical  ingredient of  the proof of Theorem~\ref{thm:main1} is a result of \cite{ELMNP1}, which we now recall.

\begin{proposition}(\cite{ELMNP1}*{Proposition 2.8})\label{prop:ELMNP}
Let $D$ be a big $\QQ$-divisor on a smooth projective variety $X$, $x\in X$ an arbitrary (closed) point. Then the following are equivalent. 
\begin{enumerate}
 \item There exists  $C>0$ having the property that  $\mult_x(|pD|)<C$, whenever $|pD|$ is nonempty for some positive integer $p$..
 \item $\mult_x(\|D\|) = 0$.
 \item $x\notin \Bminus(D)$. 
\end{enumerate}
\end{proposition}

The connection between asymptotic multiplicity and Newton--Okounkov bodies comes from the claim below. 

\begin{lemma}\label{lem:1}
Let $M$ be an integral Cartier divisor on a projective variety  $X$ (not necessarily smooth),  $s\in H^0(X,\sO_X(M))$  a non-zero global section. Then 
\begin{equation}\label{eq:1}
\ord_x(s) \ \leq \ \sum_{i=1}^{i=n} \nu_i(s),
\end{equation}
for any admissible flag $Y_{\bullet}$ centered $x$, where $\nu_{\ybul}=(\nu_1,\ldots ,\nu_n)$ is the valuation map arising from $Y_{\bullet}$.
\end{lemma}

\begin{proof}
Since $Y_{\bullet}$ is an admissible flag and the question is local, we can assume without loss of generality that each element in the flag is smooth, 
thus $Y_{i}\subseteq Y_{i-1}$ is Cartier for each $1\leq i\leq n$. 

As the local ring $\sO_{X,x}$ is regular, order of vanishing is multiplicative. Therefore 
\[
 \ord_x(s) \equ \nu_1(s)+\ord_x (s-\nu_1(s)Y_1) \dleq \nu_1(s)+\ord_x ( (s-\nu_1(s)Y_1)|_{Y_1})
\]
by the very definition of $\nu_{\ybul}(s)$, and the rest follows by induction. 
\end{proof}

\begin{remark}\label{rem:1}
Note  that the inequality in (\ref{eq:1}) is not in general  an equality for the reason that  the zero locus of $s$ might  not intersect an element of the flag 
transversally. For the simplest example of this phenomenon set  $X=\PP^2$, and take $s=xz-y^2\in H^0(\PP^2,\sO_{\PP^2}(2))$, $Y_1=\{x=0\}$ and $Y_2=[0:0:1]$. 
Then clearly $\nu_1(s)=0$, and $\nu_2(s)=\ord_{Y_2}(-y^2)=2$, but since  $Y_2$ is a smooth point of $(s)_0=\{xz-y^2=0\}$,  $\ord_{Y_2}(s)=1$ and hence 
$\ord_{Y_1}(s)<\nu_1(s)+\nu_2(s)$.
\end{remark}

For a compact convex body $\Delta\subseteq \RR^n$, we define the \textit{sum function} $\sigma:\Delta\rightarrow \RR_{+}$ by  $\sigma(x_1,\ldots ,x_n)=x_1+\ldots +x_n$. Being  continuous on a compact topological space, it takes on its extremal values.  If $\Delta_{Y_{\bullet}}(D)\subseteq \RR^n$ be a Newton--Okounkov body, then we denote the sum function by $\sigma_D$, even though it does depend on the choice of the flag $Y_{\bullet}$.

\begin{proposition}\label{prop:1}
Let  $D$ be  a big $\QQ$-divisor on a projective variety $X$ (not necessarily smooth) and let $x\in X$ a  point. Then 
\begin{equation}\label{eq:2}
\mult_x(||D||) \ \leq \ \min \sigma_D .
\end{equation}
for any admissible flag $Y_{\bullet}$ centered at $x$.
\end{proposition}

\begin{proof}
Since both sides of (\ref{eq:2})  are homogeneous of degree one in $D$, we can assume without loss of generality that $D$ is integral. 
Fix a natural number $p\geq 1$ such that $|pD|\neq \varnothing$, and  let $s\in H^{0}(X,\sO_X(pD))$ be a non-zero global section. Then
\[
\frac{1}{p}\mult_x(|pD|)  \dleq  \frac{1}{p}\ord_x(s)  \dleq  \frac{1}{p}\big(\sum_{i=1}^{i=n}\nu_i(s)\big)
\]
by Lemma~\ref{lem:1}. 

Multiplication of sections and the definition of the multiplicity of a linear series then yields
$\mult_x(|qpD|)\leq q\mult_x(|pD|)$ for any $q\geq 1$, which, after taking limits leads to 
\[
\mult_x(||D||)  \dleq  \frac{1}{p}\mult_x(|pD|) \dleq  \frac{1}{p}\big(\sum_{i=1}^{i=n}\nu_i(s)\big).
\]
Varying the section $s$ and taking into account that $\Delta_{\ybul}(D)$  is 
the closure of the set  of normalized valuation vectors of  sections, we deduce the required statement.
\end{proof}

\begin{example}\label{rem:2}
The inequality in (\ref{eq:2}) is usually strict. For a concrete example  take $X=\textup{Bl}_P(\PP^2)$, $D=\pi^*(H)+E$ and the flag $Y_{\bullet}=(C,x)$, 
where $C\in|3\pi^*(H)-2E|$ is the proper transform of a rational curve with a single cusp at $P$,  and $\{x\}=C\cap E$, i.e. 
the point where $E$ and $C$ are tangent to each other. Then 
\[
\mult_x(||D||) \equ  \lim_{p\to\infty}\Big(\frac{\mult_x(|pD|)}{p}\Big) \equ \lim_{p\to\infty}\Big(\frac{\mult_x(|pE|)}{p}\Big)  \equ 1\ .
\]
On the other hand, a direct computation using \cite{LM}*{Theorem 6.4} shows that 
\[
\Delta_{Y_{\bullet}}(D) \equ  \{ (t,y)\in\RR^2 \ | \ 0\leq t\leq \frac{1}{3}, \textup{ and } 2+4t\leq y\leq 5-5t\}\ . 
\]
As a result, $\min\sigma_D =2 > 1$. 
 
For more on this phenomenon, see  Proposition~\ref{prop:starting} below.
\end{example}

\begin{remark}
We note here a connection with functions on Okounkov bodies coming from divisorial valuations. With the notation of \cite{BKMS12}, our 
Lemma~\ref{lem:1} says that $\phi_{\ord_x} \leq \sigma_D$, and a quick computation shows that we obtain equality in the case of projectice spaces, 
hyperplane bundles, and linear flags. Meanwhile, Example~\ref{rem:2} illustrates  that $\min \phi_{\ord_x} \neq \mult_x \|D\|$
in general. 
\end{remark}

\begin{proof}[Proof of Theorem~\ref{thm:main1}] 
$(1)\Rightarrow(3)$ We are assuming  $x\notin \Bminus(D)$; let us fix an ample Cartier divisor $A$ and a decreasing sequence of real number $t_m$ such that 
$D+t_mA$ is a $\QQ$-divisor. Let $\ybul$ be an arbitrary admissible flag centered at $x$. 

Then  $x\notin \textbf{B}(D+t_mA)$ for every $m\geq 1$, furthermore,   since $A$ is ample, Lemma~\ref{lem:nested} yields
\begin{equation}\label{eq:3}
\Delta_{Y_{\bullet}}(D) \equ \bigcap_{m=1}^{\infty} \Delta_{Y_{\bullet}}(D+t_mA) \ .
\end{equation}
Because  $x\notin \textbf{B}(D+t_mA)$ holds for any $m\geq 1$,  there must exist a sequence of natural numbers $n_m\geq 1$ and a  sequence of 
global sections  $s_m\in H^0(X,\sO_X(n_m(D+t_mA)))$ such that $s_m(x)\neq 0$. This implies that $\nu_{Y_{\bullet}}(s_m)=\origin$ for each $m\geq 1$. 
In particular, $\origin\in \Delta_{Y_{\bullet}}(D+t_mA)$ for each $m\geq 1$. 
By  (\ref{eq:3})  we deduce that  $\Delta_{Y_{\bullet}}(D)$  contains the origin as well.

The implication $(3)\Rightarrow (2)$ being trivial, we will now take care of $(2)\Rightarrow (1)$. To this end assume that  $Y_{\bullet}$ 
is an admissible flag  centered at $x$ having the property that  $\origin\in\Delta_{Y_{\bullet}}(D)$, let $A$ be an ample divisor, and $t_m$ a 
sequence of positive real numbers converging to zero with the additional property that $D+t_mA$ is a $\QQ$-divisor. By Lemma~\ref{lem:nested}, 
\[
 \origin \in \Delta_{\ybul}(D) \dsubseteq \Delta_{\ybul}(D+t_mA)
\]
for all $m\geq 0$. Whence  $\min\sigma_{D+t_mA}=0$ for all sum functions  $\sigma_D:\Delta_{Y_{\bullet}}(D)\rightarrow\RR_{+}$. By Proposition~\ref{prop:1} this forces $\mult_x(||D+t_mA||)=0$ for all $m\geq 1$, hence \cite{ELMNP1}*{Proposition 2.8} leads to $x\notin\Bminus(D+t_mA)$
for all $m\geq 1$. Since Proposition~\ref{prop:ELMNP} is only valid for $\QQ$-divisors, we are left with proving the equality
\[
\Bminus(D) \equ \bigcup_m\Bstable(D+t_mA) \equ \bigcup_m\Bminus(D+t_mA)\ .
\]
The first equality comes from \cite{ELMNP1}*{Proposition 1.19}, as far as the second one goes,  $\bigcup_m\Bstable(D+t_mA) \supseteq  \bigcup_m\Bminus(D+t_mA)$
holds since  $\Bminus(D+t_mA)\subseteq \textbf{B}(D+t_mA)$ for any $m\in \NN$ according to \cite{ELMNP1}*{Exercise~1.16}. To show 
$\bigcup_m\Bstable(D+t_mA) \subseteq \bigcup_m\Bminus(D+t_mA)$, we note that  for any $t_m$ as above one finds  $t_{m+k}<t_m$ for some natural number $k\in \NN$. 
Then
\[
\Bstable(D+t_mA) \equ \Bstable(D+t_{m+k}A+(t_m-t_{m+k})A) \equ \Bminus (D+t_{m+k}A) \ ,
\] 
where the latter inclusion follows from the definition of the restricted base locus.
\end{proof}

\begin{remark}\label{rmk:non-smooth nef}
A closer inspection of the above proof reveals that the implication $(1)\Rightarrow (3)$ holds on an arbitrary projective variety both in 
Theorem~\ref{thm:main1} and Corollary~\ref{cor:nef}. 
\end{remark}

We finish  with a precise version of  Proposition~\ref{prop:1} in the surface case, which also  provides  a complete  answer 
to the question of where the Newton-Okounkov body starts in the plane. Note that unlike Theorem~\ref{thm:main3},  
it gives a full description for an arbitrary flag. 

\begin{proposition}\label{prop:starting}
Let  $X$ be  a smooth projective surface,  $(C,x)$  an admissible flag,  $D$  a big $\QQ$-divisor on $X$ with Zariski decomposition $D=P(D)+N(D)$. 
Then
\begin{enumerate}
\item $\min\sigma_D=a+b$, where $a=\mult_C(N(D))$ and $b=\textup{mult}_{x}(N(D-aC)|_C)$,
\item $\mult_x(||D||)=a+b'$, where $b'=\mult_x(N(D-aC))$.
\end{enumerate}
Moreover, $(a,b)\in\Delta_{(C,x)}(D)$ and $\Delta_{(C,x)}(D)\subseteq (a,b)+\RR^2_+$.
\end{proposition}

\begin{proof}
$(1)$ This is an immediate consequence of \cite{LM}*{Theorem 6.4} in the light of the fact that $\alpha$ is an increasing function, 
hence $\min\sigma_D$ is taken up at the point $(a,\alpha(a))$. \\
$(2)$ Since $x$ is a smooth point, it will suffice  to check that $\mult_x(||D||)=\mult_x(N(D))$. As asymptotic multiplicity is homogeneity of degree
one (see \cite{ELMNP1}*{Remark 2.3}), we can safely assume that $D, P(D)$ and $N(D)$ are all integral. 

As one has isomorphisms $H^0(X,\sO_X(mP(D)))\rightarrow H^0(X,\sO_X(mD))$ for all $m\geq 1$ by \cite{PAGI}*{Proposition 2.3.21}, 
 the definition of asymptotic multiplicity yields
\[
\mult_x(||D||) \equ \mult_x(||P(D)||) \ + \ \mult_x(N(D)) \ .
\]
Observe that  $P(D)$ is big and nef therefore \cite{PAGI}*{Proposition 2.3.12} implies  $\mult_x(||P(D)||)=0$. This completes the proof.
\end{proof}

\section{Augmented base loci}

As explained in   \cite{ELMNP1}*{Example 1.16}, one has  inclusions 
$\Bminus(D)\subseteq \textup{\textbf{B}}(D)\subseteq\Bplus(D)$,  consequently, we expect that  
whenever $x\notin\Bplus(D)$,  Newton--Okounkov bodies attached to $D$ should contain more than just the origin. 
As we shall see below, it will turn out that under the condition above they in fact contain small simplices.

We will write 
\[
\Delta_{\epsilon} \deq  \{ (x_1,\ldots, x_n)\in\RR^n_{+} \ | \ x_1+\ldots +x_n\leq \epsilon\}
\]
for the standard $\epsilon$-simplex.

Our main statement is the following.

\begin{theorem}\label{thm:main2}
Let $D$ be a big $\RR$-divisor on $X$,  $x\in X$ be an arbitrary (closed) point. Then the following  are equivalent.
\begin{enumerate}
 \item $x\notin \textup{\textbf{B}}_+(D)$.
 \item There exists an admissible flag $Y_{\bullet}$ centered at $x$ with $Y_1$ ample  such that $\Delta_{\epsilon_0} 
\subseteq \Delta_{Y_{\bullet}}(D)$ for some $\epsilon_0 >0$. 
 \item For every  admissible flag $Y_{\bullet}$ centered at $x$ there exists $\epsilon >0$ (possibly depending on $\ybul$) 
 such that $\Delta_{\epsilon}\subseteq \Delta_{Y_{\bullet}}(D)$.
\end{enumerate}
\end{theorem}

\begin{corollary}\label{cor:ample}
Let $X$ be a smooth projective variety, $D$ a big $\RR$-divisor  on $X$. Then the following are equivalent.  
\begin{enumerate}
 \item $D$ is ample.
 \item For every point $x\in X$ there exists an admissible flag $Y_{\bullet}$ centered at $x$ with $Y_1$ ample  such that $\Delta_{\epsilon_0} 
\subseteq \Delta_{Y_{\bullet}}(D)$ for some $\epsilon_0 >0$. 
\item  For every  admissible flag $Y_{\bullet}$  there exists $\epsilon >0$ (possibly depending on $\ybul$) 
 such that $\Delta_{\epsilon}\subseteq \Delta_{Y_{\bullet}}(D)$.
\end{enumerate}
\end{corollary}

\begin{proof}[Proof of Corollary~\ref{cor:ample}]
Follows immediately from Theorem~\ref{thm:main2} and \cite{ELMNP1}*{Example 1.7}. 
\end{proof}

One can see  Corollary~\ref{cor:ample} as a variant of Seshadri's criterion for ampleness in the language of convex geometry. 

\begin{remark}
It is shown in \cite{KL}*{Theorem 2.4} and \cite{KL}*{Theorem A} that  in dimension two one can in fact discard the condition above that $Y_1$ should be 
ample. Note that the proofs of the cited results  rely heavily on  surface-specific tools and in general follow a line of thought different from the present one.
\end{remark}

We first prove a helpful lemma.

\begin{lemma}\label{lem:2}
Let $X$ be a projective variety (not necessarily smooth), $A$  an ample Cartier divisor, $Y_{\bullet}$ an admissible flag on $X$. 
Then for all $m>>0$ there exist global sections 
$s_0,\ldots ,s_n\in H^0(X,\sO_X(mA))$ for which 
\[
\nu_{Y_{\bullet}}(s_0) = \origin \textup{ and } \nu_{Y_{\bullet}}(s_i)=\ei, \textup{ for each } i=1,\ldots ,n,
\]
where $\{ \eone,\ldots ,\en\}\subseteq \RR^n$ denotes the standard basis.
\end{lemma}
\begin{proof}
First, we point out that by the admissibility of the flag $\ybul$, we know that there is an open neighbourhood $sU$ of $x$ such that $Y_i|_{\sU}$ is 
smooth for all $0\leq i\leq n$. 

Since $A$ is ample,  $\sO_X(mA)$  becomes globally generated for  $m>>0$. For all such  $m$ like  there exists a non-zero section 
$s_0\in H^0(X,\sO_X(mA))$ with $s_0(Y_n)\neq 0$,  in particular,  $\nu_{Y_{\bullet}}(s_0)=\origin$, as required. 

It remains to show that for all $m>>0$ and  $i=1\leq i\leq n$ we can find  non-zero sections  $s_i\in H^0(X,\sO_X(mA))$ with 
$\nu_{Y_{\bullet}}(s_i)=\ei$. To this end, fix $i$ and let $y\in Y_i\setminus Y_{i+1}$ be a smooth point. Having chosen  $m$ large enough,  
Serre vanishing yields $H^1(X,\sI_{Y_i|X}\otimes \sO_X(mA))=0$, hence  the map $\phi_m$ in  the diagram 
\[
\xymatrix{
& & H^0(X,\sO_X(mA)) \ar[d]_{\phi_m}  \\
 0\ar[r] &  H^0(Y_i,\sO_{Y_i}(m(A|_{Y_i})-Y_{i+1})) \ar[r]^<<<<{\psi_m} & H^0(Y_i,\sO_{Y_i}(mA))  }
\]
is surjective. 

Again, by making $m$ high enough,  we can assume  $|m(A|_{Y_i})-Y_{i+1}|$ to be  very ample on $Y_i$, thus,  there will exist 
$0\neq\tilde{s}_i\in H^0(Y_i,\sO_{Y_i}(mA)\otimes \sO_{Y_i}(-Y_{i+1}))$ not  vanishing  at $x$ or $y$. Since $\tilde{s}_i(x)\neq 0$, the section
$\tilde{s}_i$  does not  vanish along $Y_j$ for all $j=i+1,\ldots ,n$. Also, the image   
$\psi_m(\tilde{s}_i)\in H^0(Y_i,\sO_{Y_i}(mA))$ of $\tilde{s}_i$  vanishes at $x$, but not at the point  $y$.

By the surjectivity of the map  $\phi_m$  there exists a section $s_i\in H^0(X,\sO_X(mA))$ such that $s|_{Y_i}=\psi_m(\tilde{s}_i)$ and $s(y)\neq 0$. 
In particular, $s_i$ does not vanish along any of the $Y_j$'s for  $1\leq i\leq j$, therefore  $\nu_{Y_{\bullet}}(s)=\ei$, as promised. 
\end{proof}

\begin{proof}[Proof of Theorem~\ref{thm:main2}]
$(1)\Rightarrow (3)$. First we treat the case when $D$ is $\QQ$-Cartier. Assume that  $x\notin \Bplus(D)$, which implies by definition that  
$x\notin\Bstable(D-A)$ for some small ample $\QQ$-Cartier  divisor $A$. Choose a positive integer  $m$ large and divisible enough  such that $mA$ 
becomes integral,  and satisfies the conclusions of  Lemma~\ref{lem:2}. Assume furthermore that $\Bstable(D-A)=\textup{Bs}(m(D-A))$ set-theoretically.  

Since $x\notin \textup{Bs}(m(D-A))$,  there exists a section $s\in H^0(X,\sO_X(mD-mA))$ not vanishing at $x$,  and
in particular $\nu_{Y_{\bullet}}(s)=\origin$. At the same time,  Lemma~\ref{lem:2} provides the existence of  global  sections 
$s_0,\ldots, s_n\in H^0(X,\sO_X(mA))$ with the property that $\nu_{Y_{\bullet}}(s_0)\equ \origin$ and $\nu_{Y_{\bullet}}(s_i)=\ei$ for all 
$1\leq i\leq n$. 

But then the multiplicativity of the valuation map $\nu_{\ybul}$ gives 
\[
\nu_{Y_{\bullet}}(s\otimes s_0)=\textbf{\textup{0}}, \textup{ and } \nu_{Y_{\bullet}}(s\otimes s_i)=\ei \ \ \text{for all $1\leq i\leq n$.}
\]
By the construction of Newton--Okounkov bodies, then $\Delta_{\frac{1}{m}}\subseteq \Delta_{Y_{\bullet}}(D)$. 

Next, let $D$ be a big $\RR$-divisor for which $x\notin\Bplus(D)$, and let  $A$ be an ample $\RR$-divisor with the property that 
$D-A$ is a $\QQ$-divisor and $\Bplus(D)=\Bplus(D-A)$. Then we have $x\notin\Bplus(D-A)$, therefore 
\[
\Delta_\epsilon\dsubseteq \Delta_{\ybul}(D-A) \dsubseteq \Delta_{\ybul}(D) 
\]
according to the $\QQ$-Cartier case and  Lemma~\ref{lem:nested}. 

Again, the implication $(3)\Rightarrow (2)$ is trivial, hence we only need to take care of $(2)\Rightarrow (1)$. 
As $Y_1$ is ample, \cite{ELMNP1}*{Proposition 1.21} gives  the equality $\Bminus(D-\epsilon Y_1)=\Bplus(D)$.  for all  $0<\epsilon <<1$. 
Fix an $\epsilon$ as above, subject to the additional condition that $D-\epsilon Y_1$ is a big $\QQ$-divisor. Then, according to 
Proposition~\ref{prop:compute},  we have 
\[
\Delta_{Y_{\bullet}}(D)_{\nu_1\geq \epsilon} \ = \ \Delta_{Y_{\bullet}}(D-\epsilon Y_1)\ + \ \epsilon \eone \ ,
\]
which yields $\origin\in \Delta_{Y_{\bullet}}(D-\epsilon Y_1)$. By Theorem~\ref{thm:main1}, this means that  
$x\notin\Bminus(D-\epsilon Y_1)=\Bplus(D)$, which completes the proof. 
\end{proof}
 
\begin{remark}\label{rmk:non-smooth ample}
The condition that $X$ be smooth can again be dropped for the implication $(1)\Rightarrow (3)$ both in Theorem~\ref{thm:main2} and 
Corollary~\ref{cor:ample} (cf. Remark~\ref{rmk:non-smooth nef}). This way, one obtains the statement that whenever $A$ is an ample $\RR$-Cartier 
divisor on a projective variety $X$,  then every Newton--Okounkov body of $A$ contains a small simplex. 
\end{remark}
  
As a consequence, we can extend \cite{KL}*{Definition 4.5} to all dimensions.   
 
\begin{definition}[Largest simplex constant]
Let $X$ be an arbitrary projective variety, $x\in X$ a smooth point, $A$ an ample $\RR$-divisor on $X$. For an admissible flag  $\ybul$ on $X$ 
centered at $x$, we set 
\[
 \lambda_{\ybul}(A;x) \deq \sup\st{\lambda>0\mid \Delta_\lambda\subseteq\Delta_{\ybul}(A)}\ .
\]
Then the \emph{largest simplex constant} $\lambda(A;x)$ is defined as 
\[
 \lambda(A;x) \deq \sup \st{\lambda_{\ybul}(A;x)\mid \text{ $\ybul$ is an admissible flag centered at $x$}}\ .
\]
\end{definition}

\begin{remark}
It follows from Remark~\ref{rmk:non-smooth ample} that $\lambda(A;x)>0$. The largest simplex constant is a measure of local positivity, and  it is known 
in dimension two that $\lambda(A;x)\leq \epsilon(A;x)$ (where the right-hand side denotes the appropriate  Seshadri constant) with strict inequality in general
(cf. \cite{KL}*{Proposition 4.7} and \cite{KL}*{Remark 4.9}).
\end{remark}

We end this section with a different characterization of $\Bplus(D)$ which puts no restriction on the flags. In what follows $X$ is again assumed to 
be smooth. 

\begin{lemma}\label{lem:aug}
For a point $x\in X$, $x\notin\Bplus(D)$ holds if and only if  
\begin{equation}\label{eq:augmented}
\lim_{p\rightarrow\infty} \mult_x(||pD-A||) \equ 0
\end{equation}
for some ample divisor $A$. 
\end{lemma}
\begin{proof}
Assuming (\ref{eq:augmented}), $x\notin\Bplus(D)$ follows from \cite{ELMNP3}*{Lemma~5.2}. For the converse implication, consider the equalities
\[
 \Bplus(D) \equ \Bminus(D-\frac{1}{p}A) \equ \textup{\textbf{B}}(pD-A)
\]
which hold for integers $p\gg 0$. Hence, if $x\notin \Bplus(D)$, then $x\notin\textup{\textbf{B}}(pD-A)$ for all $p\gg 0$. But this latter condition
implies $\mult_x(||pD-A||)=0$ for all $p\gg 0$ for all $p\gg 0$. 
\end{proof}

 
\begin{proposition}\label{prop:augmented1}
A point $x\notin\Bplus(D)$ if and only if there exists an admissible flag $Y_{\bullet}$ based at $x$ satisfying  the property that for any $\epsilon>0$ there exists a natural number $p_{\epsilon}>0$ such that 
\[
\Delta_{\epsilon} \ \bigcap \Delta_{Y_{\bullet}}(pD-A) \ \neq \ \varnothing
\]
for any $p\geq p_{\epsilon}$.
\end{proposition}
\begin{proof}
Assume first that $x\notin \Bplus(D)$. Again, by \cite{ELMNP1}*{Proposition 1.21}, we have $\Bplus(D)=\Bminus(D-\tfrac{1}{p}A)=\Bminus(pD-A)$ 
for all $p\gg 0$. Then $x\notin \Bminus(pD-A)$ for all $p\gg 0$, hence $\origin\in\Delta_{\ybul}(pD-A)$ for all $p\gg 0$ by Theorem~\ref{thm:main1}, which
implies $\Delta_{\epsilon} \cap \Delta_{Y_{\bullet}}(pD-A) \neq \varnothing$ for all $p\gg 0$. 

As far as  the converse implication goes,  Proposition~\ref{prop:1} shows that 
\[
\mult_x(||pD-A||) \dleq \min \sigma_{pD-A}.\ , 
\] 
hence the condition in the statement implies  $\lim_{p\rightarrow\infty} \mult_x(||pD-A||)=  0$. But then we are done by Lemma~\ref{lem:aug}. 
\end{proof}


\section{Nakayama's divisorial Zariski decomposition and Newton--Okounkov bodies}

In the previous sections we saw the basic connections between Newton--Okounkov bodies associated to a big line bundle $D$ and the asymptotic base loci $\Bplus(D)$ and $\Bminus(D)$.  In \cite{Nak}, Nakayama performes a deep study of these loci, he shows for instance 
that $\Bminus(D)$ can only have finitely many divisororial components. Along the way he introduces his $\sigma$-invariant, which measures the 
asymptotic multiplicity of divisorial components of $\Bminus(D)$. 

The goal of this section is to study the connection between divisorial Zariski decomposition and Newton--Okounkov bodies. 
First, we briefly recall the divisorial Zariski decomposition or $\sigma$-decomposition introduced by Nakayama \cite{Nak} and Boucksom \cite{Bou2}. 

Let $X$ be a smooth projective variety, $D$ a pseudo-effective $\RR$-divisor on $X$. Although $\Bminus(D)$ is a countable union of closed 
subvarieties, \cite{Nak}*{Theorem 3.1} shows that it only has finitely many divisorial components. 

Let $A$ be an ample divisor. Following Nakayama, for each prime divisor $\Gamma$ on $X$ we set
\[
\sigma_{\Gamma} \deq  \lim_{\epsilon\rightarrow 0^+}\textup{inf}\{\mult_{\Gamma}(D') \mid D'\sim_{\RR}D+\epsilon A \textup{ and } D'\geq 0\}\ .
\]
In \cite{Nak}*{Theorem III.1.5}, Nakayama shows that these numbers do not depend on the choice of $A$ and that there are only finitely many prime 
divisors $\Gamma$ with $\sigma_{\Gamma}(D)>0$. Write  
\[
N_{\sigma}(\Gamma)\deq \sum_{\Gamma}\sigma_{\Gamma}(D)\Gamma \textup{ and } P_{\sigma}(D)=D-N_{\sigma}(D) \ ,
\] 
and we call $D=P_{\sigma}(D)+N_{\sigma}(D)$ \emph{the divisorial Zariski decomposition} or \emph{$\sigma$-decomposition} of $D$. 
In dimension two divisorial Zariski decomposition coincides with  the usual Fujita--Zariski decomposition for  pseudo-effective divisors.

The main properties  are captured in the following statement.

\begin{theorem}\cite{Nak}*{III.1.4, III.1.9, V.1.3}\label{prop:Nakayama}
Let $D$ be a pseduo-effective $\RR$-disivor. Then
\begin{enumerate}
\item $N_{\sigma}(D)$ is effective and $\Supp(N_{\sigma}(D))$ coincides with the divisorial part of $\Bminus(D)$.
\item For all $m\geq 0$, $H^0(X,\sO_X(\lfloor mP_{\sigma}(D)\rfloor))\simeq H^0(X,\sO_X(\lfloor mD\rfloor))$.
\end{enumerate}
\end{theorem}

As  Theorem~\ref{thm:main1} describes how to determine $\Bminus(D)$ from the  Newton--Okounkov bodies associated to $D$, it is natural to wonder how we can 
compute the numbers $\sigma_\Gamma(D)$ and  $N_\sigma(D)$  in terms of convex geometry.
Relying on  Theorem~\ref{thm:main1} and  Nakayama's work, we are able to come up with a reasonable answer. 

\begin{theorem}\label{thm:main3}
Let $D$ be a big $\RR$-divisor,  $\Gamma$ a prime divisor on $X$,  $Y_{\bullet}: Y_0=X\supseteq Y_1=\Gamma\supseteq\ldots  \supseteq Y_n=\{x\}$
an admissible flag on $X$. Then
\begin{enumerate}
\item $\Delta_{Y_{\bullet}}(D) \ \subseteq \ (\sigma_{\Gamma}(D),0\ldots, 0)+\RR_+^n$,
\item $(\sigma_{\Gamma}(D),0\ldots, 0)\ \in \ \Delta_{Y_{\bullet}}(D)$, whenever $x\in \Gamma$ is a very general point.
\item $\Delta_{Y_{\bullet}}(D) = \nu_{Y_{\bullet}}(N_{\sigma}(D)) + \Delta_{Y_{\bullet}}(P_{\sigma}(D))$. Morever, 
$\Delta_{Y_{\bullet}}(D) =\Delta_{Y_{\bullet}}(P_{\sigma}(D))$, when $x\notin\Supp (N_{\sigma}(D))$.
\end{enumerate}
\end{theorem}

\begin{proof}[Proof of Theorem~\ref{thm:main3}]
For the duration of this proof we fix an ample divisor $A$. \\
$(1)$ This is equivalent to $\sigma_\Gamma(D)\leq \nu_1(D')$ for every effective $\RR$-divisor $D'\equiv D$. Fix a real number $\epsilon >0$, 
let $D''\sim_{\RR}D+\epsilon A$ is an effective $\RR$-divisor. Then 
\[
 \inf\{\mult_{\Gamma}(D')|D'\sim_{\RR}D+\epsilon A\} \dleq \mult_{\Gamma}(D'') \equ  \nu_1(D'')\ . 
\]
By Proposition~\ref{prop:definition}, this implies  the inclusion 
\[
\Delta_{Y_{\bullet}}(D+\epsilon A)\subseteq (\sigma'(D+\epsilon A),0,\ldots 0)+\RR^n_+ \ .
\]
Then   Lemma~\ref{lem:nested} and the definition of $\sigma_{\Gamma}(D)$ imply the claim.\\
$(2)$ By \cite{Nak}*{Lemma~2.1.5} we have   $\sigma_{\Gamma}(D-\sigma_{\Gamma}(D)\Gamma)=0$. Consequently, we obtain  
$\Gamma\nsubseteq\Bminus(D-\sigma(D)\Gamma)$. Because  $\Bminus(D-\sigma(D)\Gamma)$ is a countable union of subvarieties of $X$,  a very general point 
 $x$ lies outside   $\Bminus(D-\sigma_{\Gamma}(D)\Gamma)$. Theorem~\ref{thm:main1} yields $\origin\in\Delta_{\ybul}(D-\sigma_{\Gamma}(D)\Gamma)$, therefore
 the point $(\sigma_{\Gamma}(D),0\ldots, 0)$ is contained in $\Delta_{\ybul}(D)$. \\
$(3)$ Let $D_{\sigma}\sim_{\RR}P_{\sigma}(D)$ be an effective $\RR$-divisor,  then $D_{\sigma}+N_{\sigma}(D)\sim_{\RR} D$ is also an effective divisor 
for which  
\[
\nu_{Y_{\bullet}}(D_{\sigma}+N_{\sigma}(D)) \ = \ \nu_{Y_{\bullet}}(D_{\sigma})+\nu_{Y_{\bullet}}(N_{\sigma}(D)) \ .
\]
This implies the inclusion 
$\nu_{Y_{\bullet}}(N_{\sigma}(D)) + \Delta_{Y_{\bullet}}(P_{\sigma}(D))\subseteq \Delta_{Y_{\bullet}}(D)$ via Proposition~\ref{prop:definition}. 

For an effective $\RR$-divisor  $D'\sim_{\RR}D$,   \cite{Nak}*{III.1.14} gives that the divisor $D_{\sigma}=D'-N_{\sigma}(D)\sim_{\RR}P_{\sigma}(D)$ is
effective. Thus $\nu_{Y_{\bullet}}(D') \equ \nu_{Y_{\bullet}}(D_{\sigma})+\nu_{Y_{\bullet}}(N_{\sigma}(D))$, which completes the proof.
\end{proof}

Next, we study the variation of Zariski decomposition after Nakayama when varying the divisors inside the pseudo-effective cone. We start with the following lemma.
\begin{lemma}\label{lem:varying}
Suppose $D$ is a big $\RR$-divisor on $X$ and $E$-prime effective divisor. If $\sigma_E(D)=0$, then $\sigma_E(D-tE)=0$ for all $t\geq 0$.
\end{lemma}
\begin{proof}
The condition  $\sigma_E(D)=0$ implies  $E\nsubseteq \Bminus(D)$, thus,  by Theorem~\ref{thm:main1}, for a flag 
$Y_{\bullet}:X\supseteq E\supseteq \ldots\supseteq \{x\}$, with  $x\in E$ very general point, we have that 
$\origin\in \Delta_{Y_{\bullet}}(D)$. 

Again, by the very general choice of  $x\in E$, Theorem~\ref{thm:main3} says  that $\sigma_E(D-tE)\cdot \eone\in \Delta_{Y_{\bullet}}(D-tE)$. 
On the other hand, by Proposition~\ref{prop:compute} we know that 
$\Delta_{Y_{\bullet}}(D)_{\nu_1\geq t} = \Delta_{Y_{\bullet}}(D-tE)+t\eone$, therefore  $(\sigma_E(D-tE)+t)\eone\in \Delta_{Y_{\bullet}}(D)$.

By convexity, this implies $t\cdot \eone\in \Delta_{Y_{\bullet}}(D)$,  again by Proposition~\ref{prop:compute} we have 
$\origin\in \Delta_{Y_{\bullet}}(D-tE)$, hence  $\sigma_E(D-tE)=0$ by the choice of $x\in E$ and Theorem~\ref{thm:main3}. 
\end{proof}

The next proposition shows how the negative part of  the Zariski decomposition varies inside the big cone.

\begin{proposition}\label{prop:varying}
Suppose $D$ is a big $\RR$-divisor on $X$ and $E$ a prime effective divisor. Then
\begin{enumerate}
\item If $\sigma_{E}(D)>0$, then $N_{\sigma}(D-tE)=N_{\sigma}(D)-tE$, for any $t\in [0,\sigma_E(D)]$.
\item If $\sigma_E(D)=0$, then the function $t\rightarrow N_{\sigma}(D-tE)$ is an increasing function, i.e. for any $t_1\geq t_2$ the divisor $N_{\sigma}(D-t_1)-N_{\sigma}(D-t_2E)$ is effective.
\end{enumerate}
\end{proposition}
\begin{proof}
$(1)$ This statement is proved in Lemma~1.8 from \cite{Nak}.

$(2)$ Since $\sigma_E(D)=0$, then Lemma~\ref{lem:varying} implies that $\sigma_E(D-tE)=0$ for any $t\geq 0$ and in particular $E\nsubseteq \textup{Supp}(N_{\sigma}(D-tE))$ for any $t\geq 0$. So, take $\Gamma\subseteq \textup{Supp}(N_{\sigma}(D-t_2E))$ a prime divisor. The goal is to prove that $\sigma_{\Gamma}(D-t_1E)\geq \sigma_{\Gamma}(D-t_2E)$. Without loss of generality, we assume that $t_2=0$ and $t_1=t\geq 0$ and we need to show that $\sigma_{\Gamma}(D-tE)\geq \sigma_{\Gamma}(D)$. Now take a flag $Y_{\bullet}:X\supseteq \Gamma\supseteq \ldots\supseteq\{x\}$, where $x\in \Gamma$ is a very general point and $x\notin E$. Then by Theorem~\ref{thm:main3} we have
\[
\sigma_{\Gamma}(D)\cdot \eone\in \Delta_{Y_{\bullet}}(D)\subseteq \sigma_{\Gamma}(D)\cdot \eone+\RR^n_+
\]
and 
\[
\sigma_{\Gamma}(D-tE)\cdot \eone\in \Delta_{Y_{\bullet}}(D-tE)\subseteq \sigma_{\Gamma}(D-tE)\cdot \eone+\RR^n_+.
\]
On the other hand, it is not hard to see that $\Delta_{Y_{\bullet}}(D-tE)\subseteq \Delta_{Y_{\bullet}}(D)$. For any $D'\sim_{\RR}D-tE$ effective $\RR$-divisor, the $\RR$-divisor $D'+tE\sim_{\RR}D$ is also effective. Since $x\notin E$, then $\nu_{Y_{\bullet}}(D')=\nu_{Y_{\bullet}}(D'+tE)$ and the inclusion follows naturally. Combining this and the above information we obtain that $\sigma_{\Gamma}(D-tE)\cdot \eone\in \Delta_{Y_{\bullet}}(D)$ and thus $\sigma_{\Gamma}(D-tE)\geq \sigma_{\Gamma}(D)$.
\end{proof}


\raggedright

\end{document}